\documentclass[12pt]{amsart}

\usepackage{todonotes}
\usepackage[export]{adjustbox}
\usepackage{amssymb}
\usepackage{amsmath}
\usepackage{fancyhdr}
\usepackage{bbm}
\usepackage[all]{xy}
\usepackage{xcolor}
\usepackage{hyperref}
\usepackage[utf8]{inputenc}
\usepackage[T1]{fontenc}
\usepackage[english]{babel}
\usepackage[top=1in,bottom=1in,left=1in,right=1in, heightrounded]{geometry}
\hypersetup{
  colorlinks=true,
  linkcolor=blue,
  citecolor=red
}

\newtheorem{theorem}{Theorem}
\newtheorem*{theorem*}{Theorem}
\newtheorem{lemma}[theorem]{Lemma}
\newtheorem{proposition}[theorem]{Proposition}

\newtheorem{corollary}[theorem]{Corollary}

\theoremstyle{definition}
\newtheorem{definition}[theorem]{Definition}
\newtheorem*{definition*}{Definition}
\newtheorem*{lemma*}{Lemma}
\theoremstyle{remark}
\newtheorem{remark}[theorem]{Remark}
\usepackage{pstricks, enumerate, pst-node, pst-text}

\numberwithin{equation}{section}
\numberwithin{theorem}{section}

\newcommand{\PB}{\Pi}
\newcommand{\Sib}{\boldsymbol{\Sigma}}
\newcommand{\Ab}{\boldsymbol{A}}

\usepackage{xparse}
\DeclareDocumentCommand\Pr{ m g }{%
    {   \IfNoValueTF {#2}
      {\mathbb{P}\left[{#1}\right]}
      {\mathbb{P}\left[{#1}\middle\vert{#2}\right]}%
    }
}
\DeclareDocumentCommand\E{ m g }{%
    {   \IfNoValueTF {#2}
      {\mathbb{E}\left[{#1}\right]}
      {\mathbb{E}\left[{#1}\middle\vert{#2}\right]}%
    }
}
\DeclareMathOperator{\supp}{supp}

\usepackage{mathtools}

\newcommand{\norm}[1]{\left\lVert#1\right\rVert}

\title[]{Non-Realizability of the Poisson Boundary}

\author[K.\ Chawla]{Kunal Chawla}
\author[J.\ Frisch]{Joshua Frisch}

\address[K.\ Chawla]{Princeton University} 
\address[J.\ Frisch]{UC San Diego}

\thanks{This was completed during a visit of the first author to UC San Diego.}
\date{\today}

\begin{document}
\begin{abstract}
	We show that for any countable group $ G $ equipped with a probability measure $ \mu $, there exists a randomized stopping time $ \tau $ such that $ (G, \mu _{\tau} )$ admits a strictly larger space of bounded harmonic functions than $ (G,\mu) $, unless this space is trivial for all measures on $ G $. In particular, we exhibit an irreducible probability measure on the free group $F_2$ such that the Poisson boundary is strictly larger than the geometric boundary equipped with the hitting measure, resolving a longstanding open problem. As another consequence, there is never a nontrivial universal topological realization of the Poisson boundary for any countable group. 
\end{abstract}
\maketitle

\section{Introduction}

Let $ G $ be a countable group equipped with a probability measure $\mu$. We say the measure $ \mu$ is irreducible if its support generates $ G$ as a group. A function $ h: G \to \mathbb{R} $ is \emph{$ \mu$-harmonic} if $ h $ satisfies the mean-value property $ h(x) = \sum_{g \in G} h(xg)\mu(g) $. The space of bounded $ \mu $-harmonic functions can be encoded as the space of bounded measurable functions on a canonical measure space associated to the pair $ (G,\mu)$. This space is known as the \emph{Poisson boundary} of $ (G,\mu)$, which we denote by $\PB(G,\mu)$.

Probabilistically, the Poisson boundary captures stochastically significant asymptotic events of the $ \mu $-random walk, which is the Markov chain $(w_i)_{i\geq 1}$ on $ G $ with transition probabilities $ p(g,h) = \mu(g ^{-1}h)$. 
Analytically, the Poisson boundary is a measure $ G $-space with the property that for any bounded $\mu$-harmonic function $ h $ on $ G $, there exists a bounded function $ f \in L ^{\infty}(\Pi(G,\mu), \nu) $ such that \[ h(g) = \int _{\PB(G,\mu)} f(\xi) dg_*\nu(\xi) .\] This is known as the \emph{Poisson representation formula}. 

 Formally, the Poisson boundary is the quotient of the space of $ \mu $-random walk sample paths $ (G ^{\mathbb{N}}, \mathbb{P})$ by the measurable hull of the orbit equivalence relation of the time shift. That is, two paths $ (w_i)_{i\geq 1}, (w'_i)_{i \geq 1} $ are equivalent if there exists $ k,k' \in \mathbb{N} $ such that $ w_{i+k} = w'_{i+k'}$ for all $ i \in \mathbb{N}$. We point the reader to \cite{kaimanovich1996boundaries,furman2002random, zheng2022asymptotic} for more information on Poisson boundaries.
 
The Poisson boundary was popularized in Furstenberg's 1963 paper \cite{furstenberg1963poisson}, though its history goes back to the work of Blackwell \cite{blackwell1955transient}, Feller \cite{feller1956boundaries}, and Doob \cite{doob1958discrete}. It has been a powerful tool to prove rigidity theorems \cite{furstenberg1967poisson} as well as understanding the large-scale geometry of groups \cite{erschler2004boundary, Bartholdiamenability2005, 
 amirliouville2016,erschler2020growth}.
 
Furstenberg observed that in many situations this abstract measure space can be realized by a natural  space associated with the group \cite{furstenberg1963poisson}. This observation immediately leads to the \emph{topological identification problem}: given a pair $ (G,\mu) $, can one find a concrete topological ``boundary" $ B $ of the group such that (i) the $\mu$-random walk on $G$ almost surely converges in the induced bordification, and (ii) the space $B$ with the hitting measure $\nu$ is the Poisson boundary of $(G,\mu)$. We contrast this with the \emph{identification problem}, which is formulated in the measure category and asks to show that a concrete $\mu$-boundary (a priori just an equivariant quotient of the Poisson boundary) is in fact isomorphic to the Poisson boundary (see \cite{kaimanovich1996boundaries, zheng2022asymptotic}). 

A prototypical example is the case of the free group $F_2$, where for every irreducible probability measure $\mu$, the $\mu$-random walk almost surely converges to the geometric boundary of the tree $\partial F_2$. This fact is essentially due to Furstenberg \cite[Section 4]{furstenberg1963noncommuting, furstenberg1967poisson}. It is then natural to ask under what conditions on $\mu$ this boundary is a model for the Poisson boundary. 

Over forty years ago, in their 1983 paper, Kaimanovich and Vershik asked whether the Poisson boundary of the free group can be identified with the geometric boundary for \emph{all} measures. When $\mu$ is supported on the generating set of $F_2$, this was shown to be true by Dynkin-Maljutov \cite{dynkin1961random}. Derrienic \cite{derriennic1975marche} extended their result to the case of $\mu$ with finite support. In a breakthrough paper by Kaimanovich \cite{kaimanovich2000poisson}, he showed that for any $\mu$ with finite entropy and finite logarithmic moment, the Poisson boundary of $(F_2,\mu)$ is the hyperbolic boundary equipped with the hitting measure. Finally, this was extended to all finite entropy measures in joint work of the authors with Forghani and Tiozzo \cite{chawla2022poisson}. 

This question has been reiterated numerous times in the past decade \cite{forghani2015transformed, forghani2019random, erschler2023arboreal, kaimanovich2024liouville}. 

Our primary goal is to resolve this question in the negative.

\begin{theorem}\label{thm:free-group}
	    There exists an irreducible probability measure on $ F _{2} $ such that $ (\partial F _{2}, \nu) $ is not the Poisson boundary of $ (F _{2}, \mu) $.
	  \end{theorem}

The primary source of difficulty is an utter lack of general techniques in the presence of infinite entropy. The use of entropic methods in this area was pioneered by Kaimanovich and Vershik \cite{kaimanovich1983random} and has been a central motif of the field. Indeed, with a few notable exceptions, some of which we describe below, almost all Poisson boundary identification results have only applied to measures which have finite entropy. Most have used the entropic criterion of Kaimanovich \cite{kaimanovich2000poisson}. All non-entropic results instead rely on specific structures of certain probability measures \cite{blackwell1955transient, dynkin1961random, erschler2023arboreal}. On the other hand, convergence to the boundary holds for all irreducible measures, even with infinite entropy, so it is plausible that the geometric boundary may still be a model for the Poisson boundary. We show that this is not the case. Our method even extends to the free semigroup, where the random walk exhibits no backtracking (see Remark \ref{remark:free-semigroup}).





While we give a self-contained proof of Theorem \ref{thm:free-group} in Section \ref{section:free-group-proof}, it is also a corollary of the following significantly more general result, which shows that for all countable groups (except those for which the Poisson boundary is always trivial), there is no hope to find a universal boundary for all irreducible measures. That is to say, the existence of the probability measure in Theorem \ref{thm:free-group} has little to do with the structure of the free group, but is actually emblematic of a much more general phenomenon.

To state it, we require some definitions.

\begin{definition} 
	 A nonnegative integer-valued random variable $ \tau $ is a \emph{randomized stopping time} with respect to a filtration $ (\mathcal{F}_i) _{i \geq 1}  $ if there is a collection of i.i.d. Uniform$([0,1])$ random variables $ (U_i) _{i \geq 1}$ such the event $ \{ \tau \leq n \} $ is measurable with respect to the $\sigma$-algebra $ \sigma(\mathcal{F}_n, U_1, ..., U_n)$.
	  \end{definition}
 Given a probability measure $ \mu $, a random variable $ \tau $ is a \emph{randomized stopping time for the $ \mu $-random walk} if the previous definition holds with $ \mathcal{F}_n = \sigma(w_1,...,w_n)$. We denote by $ \mu_\tau $ the law of the $ \mu $-random walk stopped at the random time $\tau$. Stopping times have been a useful tool in the area since its conception, for example in work of Furstenberg \cite[Section 4.3]{furstenberg1971poisson} to relate the Poisson boundary of groups and their corecurrent subgroups, and Kaimanovich-Vershik \cite[Proposition 6.3]{kaimanovich1983random} to show triviality of the boundary for random walks on lamplighters over $\mathbb{Z}$ and $\mathbb{Z}^2$. 
 
 Randomized stopping times in particular have been used by Kaimanovich \cite{kaimanovich1992discretization} to show that Furstenberg's discretization of Brownian motion does not change the Poisson boundary \cite{furstenberg1971poisson} (see also \cite[Section 8]{lyons1984function}). Forghani and Kaimanovich prove in upcoming work  \cite{forghanikaimanovichstopping}   (though proofs are included in Forghani's thesis \cite{forghani2015transformed}) that under some mild moment assumptions on a randomized stopping time $\tau$, the pairs $(G,\mu)$ and $(G, \mu_\tau)$ have the same Poisson boundary. They also claim that standard stopping times always preserve the Poisson boundary, though there is a subtle error in the proof (see Remark \ref{remark:error}). We point out the distinction of randomized stopping times from stopping times, in that they can use randomness external to the $\mu$-random walk, and more discussion of their differences in the context of random walks on groups is available in \cite{forghani2015transformed}. There are many works exploring the relation between asymptotic properties of $(G,\mu)$ and $(G,\mu_\tau)$ under various assumptions on the stopping time $\tau$ \cite{kaimanovich1983differential,hartman2014abramov, forghani2015transformed, forghani2017asymptotic, forghani2019positive}, and we point the reader to \cite{forghanikaimanovichstopping, forghani2015transformed} for more detailed history of the subject.

We also require an algebraic property, whose relation with the Poisson boundary has been observed and exploited in numerous previous works \cite{jaworski2004countable,frisch2019choquet,erschler2023arboreal}. 

 \begin{definition} 
	   A non-trivial group $ G $ has the \emph{infinite conjugacy class property}, or ICC property, if each of its non-trivial elements has an infinite conjugacy class.
	   \end{definition}

It is a theorem of Jaworski \cite{jaworski2004countable} that if a group does not have an ICC quotient, then the Poisson boundary $ \PB(G,\mu) $ is trivial for every irreducible measure $ \mu $ (in the finitely generated case, this was previously proven by Dynkin-Maljutov \cite{dynkin1961random}). Such a group is called \emph{Choquet-Deny}. It was recently shown that for every group with an ICC quotient, there is a probability measure with a non-trivial Poisson boundary \cite{frisch2019choquet}.

A probability measure is said to be non-degenerate if $\supp \mu$ generates $G$ as a semigroup. A measure is symmetric if $\mu(g) = \mu(g^{-1})$ for all $g\in G$. 

\begin{theorem}
	Let $G$ be a countable group with an ICC quotient, and $\mu$ be any irreducible probability measure on $G$. Then there exists a randomized stopping time $\tau$ for the $ \mu $-random walk such that the Poisson boundary of $(G,\mu_\tau)$ is strictly larger than that of $(G, \mu)$. If $ \mu $ is symmetric or non-degenerate, $ \mu_{\tau} $ can be made so as well. 
\label{thm:always-bigger}
\end{theorem}

 We remark that a finitely generated group $G$ has an ICC quotient if and only if it is not virtually nilpotent \cite{frisch2018non,duguid1956fc, mclain1956remarks}, so Theorem \ref{thm:always-bigger} applies to all finitely generated groups of superpolynomial growth. In particular, there is a symmetric and non-degenerate measure on the free group whose Poisson boundary is not the geometric boundary. By Jaworski's theorem \cite{jaworski2004countable}, our Theorem \ref{thm:always-bigger} implies that unless the Poisson boundary is always trivial, one can make it strictly larger via transforming $ \mu $ by a randomized stopping time. In fact, it was explicitly asked in the thesis of Forghani whether randomized stopping times always preserve the Poisson boundary \cite[Question 4.3.1]{forghani2015transformed}; we answer this question in the negative.

 This is the first non-realizability result for the Poisson boundary. It was previously unknown whether all groups admitted a universal realization of their Poisson boundary. We not only exhibit an example where this is impossible, but show that it is never possible, unless the Poisson boundary is trivial for all measures. In other words, the identification problem has no universal solution for any group, in the sense that for any group which is not Choquet-Deny, and any irreducible measure $\mu$, there is a randomized stopping time $\tau$ such that the Poisson boundary $\PB(G, \mu)$ is a proper quotient of $\PB(G,\mu_\tau)$.

As one corollary, we show that main result of \cite{frisch2019choquet} can be achieved via the transformation of \emph{any} irreducible measure by a randomized stopping time, although this new measure may have infinite entropy.
 
\begin{corollary}
	For any countable group $ G $ with an ICC quotient equipped with an irreducible probability measure $\mu$, there exists a randomized stopping time $\tau$ such that $(G, \mu_\tau)$ has a non-trivial Poisson boundary.
\end{corollary}

\begin{remark} 
	The authors of \cite{frisch2019choquet} actually construct a measure $ \mu $ which is fully supported in the sense that $\mu(g)>0$ for all group elements $g\in G$. Examining our construction of the randomized stopping time $ \tau $ in section \ref{section:construction-of-mutau}, we see that for any fully supported measure $ \mu $, the measure $ \mu_\tau $ coming from Theorem \ref{thm:always-bigger} is also fully supported. On the other hand, our construction when applied to the free group will always produce infinite entropy measures, by \cite[Theorem 1.1]{chawla2022poisson}. We expect our general construction to always produce infinite entropy measures.
	  \end{remark}

	  As mentioned earlier, Theorem \ref{thm:always-bigger} implies the topological identification problem never has a universal solution without additional conditions on the measure. To elaborate, we define a bordification of $ G$ to be a space $ G\cup B$ where the action of $ G $ on itself extends by homeomorphisms to $G \cup B$. We say that a bordification of $ G $ is a \emph{candidate boundary} if there exists an irreducible probability measure $\mu$ such that the $\mu$-random walk on $G$ almost surely converges in $B$. We say a measure $\mu'$ is \emph{compatible} with $B$ if the $\mu'$-random walk almost surely converges to a unique point in $B$. This defines a hitting measure $\nu$ on $B$ by $\nu(U) := \mathbb{P}(\lim_{n\to \infty} w_n \in U)$.

	  \begin{corollary}
      Let $G$ be a countable discrete group with an ICC quotient. For every candidate boundary $B$, there is a compatible probability measure $\mu$ such that $B$ equipped with the hitting measure is not the Poisson boundary of $(G,\mu).$\end{corollary}
      \begin{proof}
      Since $B$ is a candidate boundary, there exists an irreducible measure $\mu$ which is compatible with $B$. If $(B,\nu)$ is not the Poisson boundary of $(G,\mu)$, we are done. If it is, let $\mu_\tau$ be the measure coming from Theorem \ref{thm:always-bigger}. Since $\mu_\tau$ is a randomized stopping time transformation of $\mu$, then the $\mu_\tau$-random walk almost surely converges and has the same hitting measure $(B,\nu)$. Since there exists a bounded $\mu_\tau$-harmonic function that is not $\mu$-harmonic, we deduce that not every bounded $\mu_\tau$-harmonic function can be realized by the Poisson integral formula on $(B,\lambda)$.
      \end{proof}

      It has been emphasized that the Poisson boundary is an inherently measure-theoretic object, and that "attempts to treat the Poisson boundary as a topological space... shroud its true nature" \cite{kaimanovich1996boundaries}. Our theorem underscores this point: the Poisson boundary never admits a topological realization for all measures.

	  As mentioned previously, results in the field so far have largely relied on entropic methods--in particular, Kaimanovich's entropy criterion for identification of the Poisson boundary \cite{kaimanovich2000poisson} (the first proof of which appeared in \cite{kaimanovich1985entropy}). These most notably include (i) groups with hyperbolic properties \cite{ledrappier1985poisson,kaimanovich2000poisson,maher2018random,horbezoutfn,chawla2022poisson} as well as (ii) wreath products and their variants \cite{erschler2011poisson, lyons2021poisson, erschler2022poisson, erschler2023poisson, silva2024poisson, vio2025poisson} (see \cite{zheng2022asymptotic,frisch2023poisson} and references therein for more examples). 
      
      There are however, a few notable exceptions. First is the original proof of Kaimanovich and Vershik that amenable groups admit irreducible measures with trivial boundary \cite[Theorem 4.4]{kaimanovich1983random}. Second is the work of Andrei Alpeev \cite{alpeev2025characterization}, who showed that if a group is not $C^*$-simple, then the action on the Poisson boundary is not essentially free for a generic measure on the group. In addition there is the work of Alpeev \cite{alpeev2021examples} and Kaimanovich \cite{kaimanovich1983examples} who showed for infinite entropy measures that the left random walk can have a trivial Poisson boundary even if the right one is nontrivial, and that the Poisson boundary need not split over products in the presence of infinite entropy \cite{alpeev2024secret,kaimanovich2024liouville}.
      Finally, we wish to highlight the recent paper of Erschler and Kaimanovich who, on any ICC group, construct a family of forests and associated probability measures such that the corresponding random walks in some sense resemble the simple random walk on a tree \cite{erschler2023arboreal}. As a result, the boundary of these forests provide models of the Poisson boundary for those particular measures.

	  We leverage Erschler and Kaimanovich's framework in the proof of Theorem \ref{thm:always-bigger}. They remark that the boundaries they construct do not have any common underlying space, and ``this absence of universality seems to be the first example of this kind". In the case of finite entropy measures on the free group, these boundaries coincide up to a conull set with the geometric boundary (see \cite{chawla2022poisson}).

      We prove that in infinite entropy, there are infinitely many non-isomorphic Poisson boundaries on the same group. To be precise, we use Erschler and Kaimanovich's framework to construct a boundary for the transformed measure $ \mu _{\tau} $, and show that this new boundary is strictly larger than the Poisson boundary of $ \mu $. We also show that this boundary is the Poisson boundary of $ \mu_\tau $ (see Corollary \ref{cor:PB-identification}). 

    We now explain the key observation behind our construction. Let $\mu$ be the measure driving the lazy simple random walk on $F_2$. If $\tau$ is a randomized stopping time for the $\mu$-random walk, then the $\mu$-random walk and the $\mu_\tau$-random walk induce the same hitting measure $\nu$ on the boundary. If $(\partial F_2, \nu)$ is the Poisson boundary of $(F_2, \mu_\tau)$, then the Poisson integral formula implies both $\mu$ and $\mu_\tau$ admit exactly the same collection of bounded harmonic functions. Hence to show that $(\partial F_2, \nu)$ is not the Poisson boundary, it suffices to exhibit a bounded function that is $\mu_\tau$-harmonic but not $\mu$-harmonic. In other words, we completely bypass the need to work with the conditional random walk or with any entropic criteria. 

	  The paper is organized as follows. In Section \ref{section:free-group-proof} we provide a self-contained proof of Theorem \ref{thm:free-group}. In Section \ref{section:switching-elements} we prove a lemma on the abundance of ``switching" elements, akin to \cite{frisch2019choquet}, \cite{erschler2023arboreal}, \cite{gorokhovsky2024quantitative}, and \cite{le2018subgroup}. In Section \ref{section:arboreal} we recall the framework of Erschler and Kaimanovich and use our lemma to construct the measure $\mu_\tau$. Finally in Section \ref{section:final-proofs} we prove Theorem \ref{thm:always-bigger} by exhibiting an explicit $\mu_\tau$-harmonic function that is not $\mu$-harmonic.

\section{Proof of Theorem \ref{thm:free-group}}\label{section:free-group-proof} 

We let $\mu$ be the measure driving the lazy simple random walk on $F_2$, that is $\mu= \frac{1}{2}\delta_e + \frac{1}{8} \sum_{s \in \{a^{\pm}, b^{\pm}\}} \delta_s$. We will construct a randomized stopping time $\tau$ for the $\mu$-random walk. 

Our randomized stopping time $\tau$ will be a mixture of a countable collection of stopping times $(\tau_i)_{i\geq 0}$, weighted by an auxiliary probability measure $p$ on the natural numbers $\mathbb{N}$. That is to say, one step of the $\mu_\tau$ random walk can be sampled by first sampling some $i \sim p$, then sampling a group element according to the stopped measure $\mu_{\tau_i}$. In particular, we can write $\mu_\tau = \sum_{i\geq 0} p(i)\mu_{\tau_i}$. For the remainder of this section, we will specify what properties we want $\tau$ to satisfy, and then we will construct $\tau$ via a judicious choice of $p$ and a sequence of stopping times $(\tau_i)_{i\geq 0}$.

\subsection{A reduction to the lamplighter group}

Recall the lamplighter group $ \mathbb{Z} \wr \mathbb{Z}/2 \mathbb{Z}$, defined as the semidirect product $\mathbb{Z} \ltimes \bigoplus _{i \in \mathbb{Z}} \mathbb{Z}/2 \mathbb{Z}   $, where the action of $ \mathbb{Z} $ is by shifting the index. Elements of this group are pairs $(X,  \varphi)$ where $ \varphi $ is a finitely supported function $ \varphi: \mathbb{Z}\to \mathbb{Z}/2 \mathbb{Z} $ and $ X $ is an integer.

We consider the homomorphism $ \pi $ from $F _{2} $ to the lamplighter $ \mathbb{Z} \wr \mathbb{Z}/2 \mathbb{Z} $ induced by setting $ \pi (a) = (e _{0}, 0) $ and $ \pi (b) = (0, 1) $, where $ e _{0}(x) = 1  $ if $ x=0 $ and $ e_0(x)=0$ otherwise.

Our choice of $ p $ and $ (\tau _{i}) _{i\geq 0} $ will be to ensure that the following proposition holds:

\begin{proposition}\label{prop:lamp-stabilization}
The Poisson boundary of $ ( \mathbb{Z} \wr \mathbb{Z}/2 \mathbb{Z}, \pi _{*}  \mu_\tau) $ is nontrivial.
	  \end{proposition}
We explain why this implies that the geometric boundary of $F_2$ is not a model for the Poisson boundary of $(F_2, \mu_\tau)$.

\begin{proof}[Proof of Theorem~\ref{thm:free-group} given Proposition~\ref{prop:lamp-stabilization}]
	Let $\nu$ be the hitting measure on $\partial F_2$ induced by the lazy simple random walk. As $ \mu_\tau$ is a randomized stopping time transformation of the lazy simple random walk, the measure $\mu_\tau$ induces the same hitting measure on $\partial F_2$. If $(\partial F_2, \nu)$ is the Poisson boundary of $(F_2, \mu_\tau)$, then the Poisson integral formula implies that $\mu$ and $\mu_\tau$ admit the same collection of bounded harmonic functions. 
    
    Hence to show the claim it suffices to exhibit a bounded $ \mu_\tau $-harmonic function that is not $ \mu $-harmonic. Since the Poisson boundary of $  (\mathbb{Z} \wr \mathbb{Z}/2 \mathbb{Z}, \pi _{*}  \mu_\tau) $ is nontrivial, we may let $ f $ be a non-constant bounded $ \pi _{*} \mu_\tau $-harmonic function on $ \mathbb{Z} \wr \mathbb{Z}/2 \mathbb{Z} $. Then the pullback $ \pi ^{*} f $ is non-constant, bounded and $ \mu_\tau $-harmonic. 
    
    We claim $\pi^*f$ is not $\mu$-harmonic. To deduce this, we show that any bounded function on $ F _{2} $ which is $ \mu $-harmonic and constant on the fibers of $ \pi $ must be constant. Indeed, we can write such a function in the form $ \pi ^{*} g $ where $ g $ is a bounded function on $ \mathbb{Z} \wr \mathbb{Z}/2 \mathbb{Z} $. Then for any $ x \in F _{2} $, we have
	\[ \sum_{y \in \mathbb{Z} \wr \mathbb{Z}/2 \mathbb{Z}} g\left( \pi (x) y\right) \pi _{*} \mu(y)
		= \sum_{\substack{y \in \mathbb{Z} \wr \mathbb{Z}/2 \mathbb{Z}, \\ k \in \pi ^{-1} (y) }} \pi ^{*} g(xk) \mu(k) 
	= \pi ^{*} g(x) = g(\pi (x)) .\]

	Hence $ g $ is $ \pi _{*} \mu$-harmonic. Since $ \pi _{*} \mu $ is a finitely supported symmetric and irreducible measure on the lamplighter,  the Poisson boundary of $ (\mathbb{Z} \wr \mathbb{Z}/2 \mathbb{Z}, \pi _{*} \mu) $ is trivial. As a result, $g$ is constant. As $\pi^*g$ is constant on the fibers of $\pi$, it must be constant everywhere.
\end{proof}

\begin{remark}
    As pointed out to us by Vadim Kaimanovich, Proposition \ref{prop:lamp-stabilization} implies Theorem \ref{thm:free-group} by the fact that for any normal subgroup $N$ of a countable group $G$, the Poisson boundary $\PB(G/N, \pi_* \mu)$ is the space of ergodic components of the action of $N$ on $\PB(G, \mu)$ (see \cite{kaimanovichabstract1982} or \cite[Theorem 2.4.2]{forghani2015transformed}). Hence if the Poisson boundary of $(\mathbb{Z} \wr \mathbb{Z}/2 \mathbb{Z}, \pi_*\mu_{\tau})$ is larger than that of $(\mathbb{Z} \wr \mathbb{Z}/2 \mathbb{Z}, \pi_*\mu)$, then the Poisson boundary of $(F_2, \mu_{\tau}$) is strictly larger than the geometric boundary. For the convenience of the reader, we decided to retain the proof above. 
\end{remark}

We now proceed with the construction of $ p $ and $ ( \tau _{i}) _{i \geq 0}  $.

\subsection{A heavy-tailed probability measure on $\mathbb{N}$} 
Given a probability measure $ p $ on $ \mathbb{N} $, let $ (X _{i}) _{i \geq 1} $ be i.i.d samples drawn according to $ p $. Let $ T _{k} $ be the $ k $th \emph{record time} given by $ T _{0} = 1, T _{k} = \inf \{ i > T _{k-1}, X _{i} \geq X _{j} \forall j < i \} $. In addition, let $R_k = X_{T_k}$ be the $k$th record value.

\begin{definition}
    A probability measure $ p$ on $\mathbb{N}$ has \emph{eventually simple records} if almost surely, for $ k $ sufficiently large, we have $ R_{k+1}> R_k $. 
\end{definition}

We record some lemmas about records times of i.i.d. samples.

\begin{lemma} \label{lemma:simple-records}\cite[Lemma 2.3]{vervaat1973limit,  brands1994number,qi1997note,eisenberg2009number,erschler2023arboreal}
	There exists a probability measure $ p $ on $ \mathbb{N} $ which satisfies $p(i)>0$ for all $i \geq 0$ and has eventually simple records.
	  \end{lemma}
	  \begin{remark} 
		  In fact, there are many probability measures with eventually simple records. It was observed by Vervaat (see \cite{vervaat1973limit}) that the sequence of record values $ (R _{k}) _{k\geq 0} $ is a Markov chain with transition probabilities $ p(i,j) := p(j) / \sum_{\ell \geq i} p(\ell) $. By Borel-Cantelli, a sufficient condition for eventual record simplicity is that $ \sum _{i \geq 0} p(i,i) ^{2} < \infty$ and that $ p $ has infinite support. This is satisfied, for instance, by the measure $ p(n) = c(n+1) ^{-2}$ for an appropriate constant $ c>0$. 
		    \end{remark}
		    
	  For the rest of this section, fix a measure $p$ satisfying the conclusion of Lemma \ref{lemma:simple-records}. 

	  We also introduce a gauge function which controls the record times in terms of previous record values.

	  \begin{lemma}\label{lemma:gauge-ii}\cite[Lemma 2.17]{erschler2023arboreal}
	For any probability measure $ p $ on $ \mathbb{N} $, there exists a nondecreasing $ \Phi: \mathbb{N}\to \mathbb{N} $ such that almost surely for $ k $ sufficiently large we have $ T _{k+1} < \Phi(R_k) $.
	  \end{lemma}
We note that the condition here is on $T_{k+1}$, not $T_k$. 

\subsection{A sequence of stopping times} 
Let $ Z _{n} $ denote the simple random walk on $ F _{2} $ and $ (X_, \varphi _{n}) = \pi (Z _{n}) $ its projection to $ \mathbb{Z} \wr \mathbb{Z}/2 \mathbb{Z} $. Given sequences of natural numbers $ (s _{n}) _{n \geq 0}, (r _{n})_{n\geq 0} $, we define a sequence of stopping times $ (\tau _{n}) _{n \geq 0} $ by
\[ \tau _{n} = \inf \left\{ t \geq 1, \varphi_t|_{[-s _{n}, s _{n}]} \equiv 0, |X_t| \geq r _{n}  \right\} .\] 

In other words, the first time that all the the lamps in the interval $[-s_n,s_n]$ are off and the lamplighter is more than $r_n$ away from the origin. 

Since the simple random walk on $ \mathbb{Z} $ is recurrent, it follows that each $ \tau _{n} $ is almost surely finite. Indeed, recurrence of the simple random walk implies the recurrence of the pair $(\varphi_t|_{[-s_n,s_n]}, X_t)$, so that almost surely the condition in the definition of $\tau_n$ is met infinitely often.


We proceed to inductively make a choice of $ s _{1} < r _{1} < s _{2} < r _{2} < ... $. Set $ s _{1} = r _{1} = 0 $. Let $ r _{1},s _{1},..., s _{k-1}, r _{k-1} $ be given. Given a sequence $ (i_1, i_2, ... i_{\Phi(k)}) \subset \{1,..., k-1\} ^{\Phi(k)}$, consider the random sequence $ (g _{j} )_{1 \leq i \leq \Phi(k)} $ where the $ g _{j} $ are each drawn independently from $\mu_{\tau^{(j)}}$, the transformation of $\mu$ by the stopping time $ \tau^{(j)}=\tau _{i _{j}} $, which is well defined as each element of the sequence is less than $ k$ and we have already chosen $r_1,s_1,..., r_{k-1}, s_{k-1}$. Then for $ 1 \leq j \leq \Phi(k) $, let $ w_j = g_1...g_{j} $, and let $ \varphi_j $ be the lamp configuration of $ \pi (w_j) $. 

We pick $ s_k $ large enough that for every choice of sequence $ (i_1, ..., i_{\Phi(k)}) $, with probability at least $1-2^{-k}$ the support of $ \varphi_j $ is contained in $ [-s_k/3, s_k/3] $ for all $ 1\leq j \leq \Phi(k) $. Then we set $ r_k = 3s_k $. 

With our sequence of stopping times and probability measure $p$ chosen,  we are now ready to prove proposition \ref{prop:lamp-stabilization} and hence Theorem \ref{thm:free-group}.

\begin{proof}[Proof of Proposition~\ref{prop:lamp-stabilization}]
	First observe that $\mu_\tau$ is irreducible as $p(0)>0$ and $\mu$ is irreducible. We show the Poisson boundary of $(\mathbb{Z}\wr\mathbb{Z}/2\mathbb{Z}, \pi_*\mu_\tau)$ is nontrivial by showing that the sequence of lamp configurations $ (\varphi _{j}) _{j \geq 0} $ almost surely converges. 
	Let $ (X_i) _{i\geq 1}  $ be i.i.d samples from $ p $, and $ (T _{k}) _{k \geq 1} $ the associated record times.
	Then almost surely there exists $ k_0 $ such that for $ k\geq k_0 $, we have $R_{k-1} < R_k < R_{k+1}$, and $T_{k+1} < \Phi(R_k)$ (by Lemma \ref{lemma:gauge-ii}).
	Given some $ k \geq k_0 $, consider the random walk at time $ T _{k+1}-1,$  

	\[ w _{T _{k+1} - 1} = \underbrace{g _{1} ... g _{T _{k} -1}}_{a} g _{T _{k}} \underbrace{g _{T _{k} +1} ... g _{T _{k+1} -1}}_{b}   .\] Observe that each of $ a $ and $ b $ is a product of at most $ \Phi(R_k) $ increments, drawn independently from some sequence $ \mu_{\tau_{i_1}}, ..., \mu_{\tau_{i_\ell}} $ where $ i_j < R_k$ for all $ 1 \leq j \leq \ell $. 
	Then by our choice of $ (s_k)_{k\geq 0}, (r_k)_{k \geq 0} $ together with the Borel-Cantelli lemma, we know that almost surely for $ k$ sufficiently large, we have $ $ 
		  \[ \varphi _{T _{k} -1} | _{[- s _{R _{k}}/3, s _{R _{k}}/3]} = \varphi _{j}| _{[-s _{R _{k}}/3, s _{R _{k}}/3]} \] for all $ T_k \leq j \leq T _{k+1} -1 $. Since the records $ R _{k} $ are by definition increasing and they almost surely go to infinity, the limit $ \lim _{j \to \infty} \varphi _{j} $ exists almost surely.

          As $\mu_\tau$ is irreducible, the limiting lamp configuration is not almost surely constant, so the Poisson boundary of $ ( \mathbb{Z} \wr \mathbb{Z}/2 \mathbb{Z}, \pi _{*}  \mu) $ is nontrivial.
\end{proof}

\begin{remark}\label{remark:free-semigroup}
    We remark that this construction can also be made to work for the free semigroup. For the semigroup $F_2^+$ generated by $a$ and $b$, we get a surjection onto the lamplighter by sending $a$ to $(0,1)$ and $b$ to $(e_0, -1)$. The simple random walk on the free semigroup pushes forward to a irreducible random walk on the lamplighter whose projection to $\mathbb{Z}$ is recurrent, so a similar choice of randomized stopping time works to produce a measure on $F_2^+$ with Poisson boundary strictly larger than the geometric boundary.
\end{remark}

\begin{remark}
    We can make our construction work with a non-randomized stopping time as opposed to a randomized stopping time. To do this, we exploit the randomness coming from the number of increments of the $\mu$-random walk which take the value $e$. Indeed, observe that the supports of $\left(\mu_{\tau_i}\right)_{i\geq 0}$ are pairwise disjoint, so that we may let $\tau$ be the stopping time given by stopping the $\mu$-random walk the first time it lies in the support of some $\mu_{\tau_n}$ if the number of identity increments is at most $\lfloor 1/p(n) \rfloor$ modulo $\lfloor 1/p(n) \rfloor^2$. If $(s_k)_{k\geq 1}, (r_k)_{k\geq 1}$ are sufficiently large depending on $p$, then the number of identity increments modulo $\lfloor 1/p(n)\rfloor^2$ is asymptotically uniformly distributed and asymptotically independent of the position of the random walk, so that our same proof works to show that the induced random walk on the lamplighter has non-trivial boundary.
\end{remark}

\begin{remark}\label{remark:error}
    It is claimed in an unpublished manuscript of Forghani and Kaimanovich \cite{forghanikaimanovichstopping}, and quoted in \cite[Theorem 3.6.1]{forghani2015transformed}, that non-randomized stopping times do not modify the Poisson boundary. However there is a subtle error in the proof. The transformed random walk on the group $G$ is lifted to the free semigroup on the set $\supp \mu$, viewed as abstract symbols, where the Poisson boundary is left unchanged by a lift of the stopping time. The transformed measure on the free semigroup, however, is no longer generating, so one can not necessarily conclude that the Poisson boundary of the quotient $(G, \mu_\tau)$ is left unchanged.
\end{remark}

\section{Random walks and switching elements in ICC groups}\label{section:switching-elements}

Given a finite subset $ F \subset G $ and a (not necessarily finite) subset $ A \subset G $, we say that $ A $ is \emph{$ F $-switching} if for any $ f_1,f_2,f_3,f_4 \in F, a_1,a_2 \in A $, the equality $ f_1a_1f_2 = f _{3} a_2 f_4$ implies $ f_1=f_3 $, $ f_2=f_4 $, and $ a _{1} = a_2 $. We say that an element $ g \in G $ is \emph{$F$-switching} if the set $ \{ g \}  $ is $ F$-switching and that $ g $ is \emph{$ F $-superswitching} if the set $ \{ g, g ^{-1} \} $ is $ F $-switching. Switching sets were defined in \cite{erschler2023arboreal}, generalizing the notion of switching elements from \cite{frisch2019choquet, frisch2019strong},

The goal of this section is to prove Lemma \ref{lemma:stopping-at-switcher} where we produce, for any finite $ F \subset G $, a stopping time $ \tau $ such that the support of $ \mu _{\tau} $ is $ F $-switching. This lemma is key to the construction of our randomized stopping time in the proof of Theorem $ \ref{thm:always-bigger} $, and is applied in Section \ref{section:construction-of-mutau}.

Without loss of generality we may assume that $ \mu(e) > 0 $, since replacing $ \mu $ with $ \frac{1}{2} \mu + \frac{1}{2} \delta _{e} $ preserves the space of bounded harmonic functions. One may interpret this as first transforming $ \mu $ by a randomized stopping time which is $ 0 $ or $ 1 $ each with probability $ 1/2 $, then transforming it by the randomized stopping time $\tau$.

The authors of \cite{frisch2019choquet} and \cite{erschler2023arboreal} construct switching and superswitching elements using a density argument, that if a group is ICC then the set of non-$F$-superswitching elements is small in an appropriate sense, and so there are infinitely many $F$-superswitching elements. However, to construct a measure using a randomized stopping time, it is not enough to show the existence of switching elements. Instead we need to show that the $ \mu $-random walk witnesses an abundance of switching elements. To do this, we take an approach inspired by \cite{gorokhovsky2024quantitative}. Roughly speaking, we show that any irreducible random walk run for sufficiently long will be $F$-switching with high probability. We need the following lemma, similar to \cite[Claim 2.1]{gorokhovsky2024quantitative} and \cite[Lemma 3.2]{le2018subgroup}.

\begin{lemma} \label{lemma:Linf-decay}
	For any irreducible probability measure $ \mu $ on a countable group $ G $, any subgroup $ H < G $ of infinite index and finite subset $ F \subset G $, we have $ \mathbb{P}(w _{n} \in HF ) \to 0 $.
	  \end{lemma}
\begin{proof}
	Consider the action of $ G $ on the space of right cosets $ G /H$. We first claim that every $ \supp \mu $-semigroup orbit is infinite. Indeed, if there was a finite set $ X $ such that $ Xg \subset X $ for all $ g \in \supp \mu $, then we would have $ Xg = X $ and so $ X = X g ^{-1} $. Therefore the group generated by $ \supp \mu $ would fix $ X $, which would imply $ H $ has finite index.

	Now consider the (directed) transition graph of the Markov chain driven by the $ \mu$-random walk on the space of right cosets. For the claim to fail, there must exist some $Hf$ contained in some strongly irreducible component $ \mathcal{S}$. Since the $ \supp \mu $-semigroup orbit is infinite, the counting measure on $ \mathcal{S} $ is an infinite stationary measure, so the ergodic theorem for aperiodic and irreducible Markov chains implies $ \mathbb{P}(w _{n} \in Hf) \to 0$. The Markov chain is aperiodic as we assumed $\mu(e)>0$.

	Applying the union bound over the finitely many $f\in F$ gives the claim.
	\end{proof}

Using this lemma, we show that random walks are $ F $-switching with high probability.

\begin{lemma}\label{lemma:RW-superswitching}
	For any irreducible probability measure $ \mu $ on a countable ICC group $ G $ and finite subset $ F \subset G$, we have 
	\[ \mathbb{P}(w_n \text{ is } \text{$F$-switching}) \to 1.\] If $\mu$ is non-degenerate, then
	\[ \mathbb{P}(w_n \text{ is } \text{$F$-superswitching}) \to 1.\] 
\end{lemma}
\begin{proof}
	Without loss of generality we may assume $ F $ is symmetric, as any element that is $ (F\cup F ^{-1})$-switching must be $ F$-switching. For an element $ a $ to not be $F$-switching, we must have $ f_1af_2=f_3af_4 $ for some $ (f_1,f_2) \neq (f_3,f_4) \in F \times F $, hence $ af_2f_4^{-1}a^{-1} = f_1^{-1}f_3 $. Given any non-trivial element $ f \in F^2 $, the collection of $ a \in G $ such that $ a f a ^{-1} \in F^2$ is contained in a union of at most $ |F|^2 $ cosets of the centralizer $ C _{G} (f) $. As $ G $ has the ICC property, the subgroup $ C _{G} (f) $ has infinite index, so Lemma \ref{lemma:Linf-decay} implies $ \mathbb{P}(w_nfw_n^{-1} \in F^2) \to 0 $. Applying the union bound over $ f \in F^2 \setminus \{ e \} $ we see that $ \mathbb{P}(w_n \text{ is } F\text{-switching}) $ tends to 1. 

	Now suppose that $\mu$ is non-degenerate. 
	For $ x,y \in F $, let $ S _{x \to y} = \{ g \in G, gxg = y \}  $, so by the union bound we are done if we show that $ \mathbb{P}(w _{n} \in S _{x \to y}) \to 0 $ for all $ (x,y) \neq (e,e) $.

	Suppose for the sake of contradiction that $ \limsup \mathbb{P}(w _{n} \in S _{x,y}) > 0   $ for some $ (x,y) \neq (e,e) $. Let $ g \in S _{x,y} $, then any element $ h $ contained in $ g ^{-1} S _{x,y}$ must be in the centralizer of $ (xg) ^{2}  $ (see \cite[Proposition 2.5]{frisch2019choquet}). Since $ \limsup \mathbb{P}(w_n \in S _{x,y}) > 0 $, then as $ g $ is in the semigroup generated by $\supp \mu$ we deduce that $ \limsup \mathbb{P}(w_n \in g ^{-1} S _{x,y}) > 0 $ as well. Hence $ \limsup \mathbb{P}(w_n \in C _{G}(\left(xg\right)^2)) > 0$ so we see from Lemma \ref{lemma:Linf-decay} that $ C _{G} \left((xg) ^{2}\right)  $ must have finite index. Since $ G $ is ICC, this implies $ (xg) ^{2} = e $. Letting $ I $ be the set of involutions, we see that $xS_{x,y} \subset I$ so in particular $\limsup \mathbb{P}(w_n \in I) > 0$. However, since $ G $ is ICC, it is in particular not virtually abelian, so by Theorem 5.1 of \cite{amir2023probabilistic}, we have $ \mathbb{P}(w_n \in I) \to 0 $, a contradiction.
\end{proof}

It would be tempting to construct $ \mu_\tau $ by letting $ \tau $ be the first time that the $ \mu $-random walk is $ F $-switching. This would produce a measure $ \mu _{\tau}  $ such that every element of $ \supp \mu _{\tau} $ is $ F$-switching. However, we need to produce a measure $ \mu _{\tau} $ such that the \emph{set} $ \supp \mu _{\tau} $ is $ F $-switching, which is a strictly stronger property. We construct such a measure in the following lemma.

\begin{lemma}\label{lemma:stopping-at-switcher}
	Let $G$ be a countable ICC group. For any finite set $ F \subset G$ and irreducible probability measure $ \mu $ on $ G $, there exists an almost surely finite (non-randomized) stopping time $ \tau $ with the property that $ \supp \mu _{\tau}$ is $F$-switching. Moreover if $\mu$ is symmetric, then $ \mu _{\tau} $ can be made symmetric.
	  \end{lemma}
\begin{proof}
	Without loss of generality we may assume $ F $ is symmetric, as any set which is $ (F\cup F ^{-1})$-switching is $ F $-switching. Suppose $ \mu $ is irreducible (resp. irreducible and symmetric). We first `thin' out the set of all $F$-switching (resp. superswitching) elements as follows. Let $( U _{x}) _{x \in G} $ be a collection of i.i.d. Uniform$\left([0,1]\right)$ random variables, and let $ T $ be the subset of $ G $ consisting of all $ a $ that are $ F $-switching (resp. $F$-superswitching) and satisfy $ U_a \geq U _{b} $ for all $ b \in F^2aF^2$ (resp. $F^2a^{\pm}F^2$). Almost surely the random variables are distinct, so the set $T$ (resp. $ T \cup T ^{-1} $) is $ F $-switching. We let $ \tau $ be the first hitting time of $T$ (resp. $ T \cup T ^{-1} $), so it suffices to show that $ \tau $ is almost surely finite. 

	By Lemma \ref{lemma:RW-superswitching}, almost surely for every sample path of the $ \mu $-random walk, $ w _{n} $ is $F$-switching infinitely often. If $ \mu $ is irreducible and symmetric, it is non-degenerate, so $ w_n $ is $F$-superswitching infinitely often. Observe that the collection of subsets $ \{ F^2aF^2, a \text{ is } F\text{-switching} \}$ (resp. $ \{ F^2 a ^{\pm} F^2, a \text{ is } F\text{-superswitching} \} $) induces a partition of the set of $ F $-switching (resp. $F$-superswitching) elements into sets of size at most $ 2|F| ^{4} $. Then the sample path will almost surely intersect infinitely many distinct such subsets $ A _{1}, A _{2},... $ in elements $ a _{1}, a _{2}, ...$ . Since the random variables $ (U _{x} ) _{x \in G} $ are independent of the random walk, then each $ a _{i} $ has probability at least $|F| ^{-4}/2 $ of being in the random set $ T $, independent of all other $ a _{j} $. Letting $ \tau $ be the first hitting time of $T$ (resp. $ T \cup T ^{-1} $), we see that $ \tau $ is almost surely finite. 

	We observe that if $ \mu $ is symmetric then, as $ T \cup T ^{-1} $ is symmetric, the stopped measure $ \mu _{\tau} $ is clearly symmetric as well. By construction the set $ \supp \mu _{\tau}$ is $ F $-switching.
\end{proof}

\section{Arboreal structures on groups}\label{section:arboreal}

In this Section we construct our randomized stopping time used in the proof of Theorem \ref{thm:always-bigger}. Using the framework developed in \cite{erschler2023arboreal}, we produce a $ \mu _{\tau} $-boundary which, in section \ref{section:final-proofs}, enables us to construct a $ \mu _{\tau} $-harmonic function that is not $ \mu $-harmonic. A consequence of the framework of \cite{erschler2023arboreal} is a short proof that this $ \mu _{\tau} $-boundary is the Poisson boundary of $ (G, \mu _{\tau}) $ (see Corollary \ref{cor:PB-identification}).

We first recall some notions from \cite{erschler2023arboreal}.

\begin{definition} 
	A \emph{scale} on countable group is a triple $ \Lambda = ( \lambda, \Sib, \Ab) $ where $ \lambda: \mathbb{N} \to \mathbb{N} $ is a non-decreasing function, and $ \Sib = \left(\Sigma_1, \Sigma_2, ...\right) $ and $ \Ab = (A _{0}, A _{1},... ) $ are two sequences of subsets, where $ \left(\Sigma _{n}\right) _{n \geq 1}$ are pairwise disjoint and $ A _{0}$ contains the identity.
	  \end{definition}

  Given a scale, we define \[ \Delta _{n} = A _{0} \cup \bigcup _{i < n} \\\left(\Sigma_i ^{\pm}\cup A _{i} ^{\pm }\right).\] 

  Observe that in the definition of $ \Delta _{n},$ the union is taken over $ i $ strictly less than $ n $, and that $ \Delta_n$ contains the identity. 
\begin{definition} 
	A \emph{spike decomposition} of a group element $ g \in G $ with respect to a scale $ (\lambda, \Sib, \Ab) $ is a triple $( \overleftarrow{g}, \widehat{g}, \overrightarrow{g}) $ with $ g= \overleftarrow{g}\widehat{g}\overrightarrow{g} $, where $ \widehat{g} \in \Sigma _{n}  $ for some $ n$, and the prefix and postfix $  \overleftarrow{g},  \overrightarrow{g}  $ are products of at most $ \lambda (n) $ elements from $ \Delta _{n} $.
	  \end{definition}

\begin{definition} 
	The \emph{despiking graph} of a scale $ \Lambda $ is the graph with vertex set $ G $ and edges $ (\overleftarrow{g}, g) $ where $ \overleftarrow{g} $ is the prefix of a spike decomposition of an element $ g \in G $.  
	  \end{definition}

The following definition is slightly stronger than that in \cite{erschler2023arboreal}, but suffices for our purposes.
\begin{definition}
	We say a scale $ \Lambda =(\lambda, \Sib, \Ab)$ is a \emph{ladder} if \begin{enumerate} 
		\item Any $ g\in G  $ admits at most one spike decomposition.
		\item For all $n\geq 1$, the sets $ \Sigma _{n}$ and $ \Delta_n^{3 \lambda(n)}$ are disjoint.
		  \end{enumerate}
	  \end{definition}

\begin{proposition} 
	  For any ladder $ \Lambda $ on a group $ G$, 
	  \begin{enumerate} 
		  \item The despiking graph $ \mathcal{F}  $ is a forest.
		  \item Any tree of the forest $ \mathcal{F}  $ contains a unique unspiked vertex
		    \end{enumerate}
	  \end{proposition}

	  We root each tree in $ \mathcal{F} $ at its unique unspiked vertex, and refer to $ \mathcal{F} $ as the \emph{ladder forest} associated to $ \Lambda $. 

	  The following lemma is the raison d'\^{e}tre of arboreal structures.
\begin{lemma}\cite[Lemma 4.5]{erschler2023arboreal}\label{lem:switching-ladder}
	If $ \Lambda = (\lambda, \Sigma, A) $ is a scale on a countable group $ G $ such that for all $ n \geq 1 $, $ \Sigma_n $ is $ \Delta _{n} ^{5 \lambda(n)}$-switching, then $ \Lambda $ is a ladder.
	  \end{lemma}

\subsection{Construction of $\mu_\tau$}\label{section:construction-of-mutau}

Let $G$ be an ICC group equipped with an irreducible probability measure $\mu$.

By Lemma \ref{lemma:simple-records}, we fix a probability measure $ p $ on the natural numbers $ \mathbb{N} $ such that $ p(i) > 0 $ for all $ i \in \mathbb{N} $ and records are eventually simple. We also recall our gauge function $\Phi$ from Lemma \ref{lemma:gauge-ii}.

We inductively define a sequence of stopping times $ \tau _{i}$ for $ i \geq 0 $. First we let $ \tau _{0} = 1 $, and given $ \tau _{0}, ..., \tau _{k-1} $ we define $ \tau _{k} $ as follows. We inductively keep track of a collection of finite subsets $ \left(S _{i,j}\right) _{i,j \geq 1} $ of $ G $ such that $ \mu _{\tau _{i}} (S _{i,j}) > 1-\Phi(j) ^{-1} 2 ^{-j}p_i $ and $ S _{i, j-1} \subseteq S _{i,j}$. We define this collection by choosing $ \left(S_{i,k}\right) _{i\geq 1}$ after choosing $ \tau _{k}$. Also let $ S _{0,j} = \{ e \} $. 

	We let $ \tau _{k} $ be the stopping time coming from Lemma \ref{lemma:stopping-at-switcher} taking input the finite set 
\[ \left(\bigcup _{0 \leq i \leq k-1} S _{i,k} \right) ^{5\Phi(k)}.\] 

Denote by $ A _{i} $ the set inside the parentheses in the previous expression. Letting $ \Sigma_i := S_{i,i}$, we see from Lemma \ref{lem:switching-ladder} that the scale $ (\Phi, \Sib, \Ab) $ is a ladder.

We define a randomized stopping time $\tau$ as a mixture of the stopping times $(\tau_i)_{i\geq 0}$ weighted by $p$. In other words, we have \[ \mu _{\tau}= \sum_{i\geq0} p(i) \mu _{\tau _{i}}.\] 

\subsection{Convergence to the boundary}

Since the scale $ (\Phi, \Sib, \Ab) $ is a ladder, by Lemma \ref{lem:switching-ladder} we can consider the associated ladder forest $\mathcal{F}$. We show that the $\mu_\tau$-random walk almost surely converges to the boundary of this forest. This is a very slight modification of the proof of \cite[Theorem 3.1]{erschler2023arboreal}, included mainly for the convenience of the reader.

\begin{proposition}\label{prop:convergence} 
	For any $ g \in G $, the $\mu _{\tau}$-random walk started at $ g $ almost surely converges to the boundary of the forest $ \partial \mathcal{F} $. In addition, almost every sample path visits all but finitely many points in the corresponding geodesic ray in $ \mathcal{F} $.
	  \end{proposition}

\begin{proof}
	We sample the $ \mu_\tau $-random walk in the following way. First we draw $ (X _{i}) _{i \geq 1}  $ i.i.d. according to $ p $. Then, conditional on the $ X _{i}$'s, we sample $ g _{i}\sim \mu _{\tau _{X_i}} $. By Lemmas \ref{lemma:simple-records} annd \ref{lemma:gauge-ii}, there almost surely exists $ k _{0} $ such that for all $ k \geq k_0 $, the $ k $th record time satisfies $ T _{k+1} \leq \Phi(R_{k})$ and the record $ R_k $ is simple.

       Let $ E_k$ be the event that for every $ 1 \leq i \leq \Phi(k) $, the increment $ g _{i}  $ lies in $ S _{X_i, k}$. Then we have $ \mathbb{P}(E_k ^{c}) \leq 2^{-k} $. In addition, $ g_{T_k}$ lies in $ \Sigma_{R_k}$ with probability at least $ 1-p_{R_k}2^{-R_k} $. By Borel-Cantelli, almost surely for all $ k $ sufficiently large the event $ E _{k} $ holds and $ x_{T_k} \in \Sigma_{R_k}$. Increasing $ k_0 $ if necessary, suppose this holds for all $ k \geq k_0$. 

	Examine the random walk at time $n$ satisfying $ T_k \leq n \leq T_{k+1}-1$.
	\[ w _{n} = \underbrace{g _{1} ... g _{T _{k} -1}}_{a} g _{T _{k}} \underbrace{g _{T _{k}+1} ... g _{n}}_{b}   .\] 

	Then each of $ a $ and $ b $ is a product of at most $ \Phi(R_k) $ elements from $ \Delta_{R_k}$, hence $ a g _{T _{k}} b$ is a spike decomposition of $ w_n $ with spike $ g _{T _{k}} $ and prefix $ a = w _{T_k-1}$. As this holds for all $ k \geq k_0 $, the sequence $ \left(w _{T_k-1}\right) _{k \geq k_0}$ is a geodesic ray in $ \mathcal{F}$, $ w_n $ converges to the boundary of $ \mathcal{F} $ and $ w_n $ intersects all but finitely many points of $\left(w _{T_k-1}\right) _{k \geq k_0}$. 

	Now suppose we start the random walk at $ g \in G $. Since $ g \in \Delta_\ell $ for $ \ell $ sufficiently large, and almost surely infinitely many of the increments $ g_i $ are equal to the identity as $ \mu(e), p(0) > 0 $, the previous argument shows that increment $ gw_n = gag_{T_k}b$ is still spiked with prefix $ ga $ and spike $ g_{T_k}$ for $ k $ sufficiently large.   
\end{proof}

Using the ``trunk criterion" of Erschler and Kaimanovich, we can show the following 

\begin{corollary}\label{cor:PB-identification}
	Let $ \nu $ be the hitting measure on the boundary of the forest $ \mathcal{F}  $. Then $ (\partial \mathcal{F}, \nu) $ is a model for the Poisson boundary of $ (G, \mu _{\tau})$. 
	  \end{corollary}
\begin{proof}
	By Proposition \ref{prop:convergence}, for $ \nu $-a.e. $ \xi \in \partial \mathcal{F}$, the conditional random walk almost surely hits all but finitely many points on the unique geodesic in $ \mathcal{F} $ converging to $ \xi $. In the language of \cite{erschler2023arboreal}, this implies the geodesic forms a trunk for the conditional random walk, so by \cite[Proposition 3.8]{erschler2023arboreal}, the space $ (\partial \mathcal{F}, \nu )$ is the Poisson boundary of $ (G, \mu _{\tau}) $. 
\end{proof}

\section{Proof of Theorem \ref{thm:always-bigger}}\label{section:final-proofs} 

Let $G$ be a countable group with an ICC quotient, equipped with an irreducible probability measure $\mu$. Without loss of generality, we may assume that $ G $ itself has the ICC property. Indeed, any group with an ICC quotient carries a canonical normal subgroup $H$, known as the hyper-FC-Centre, with the property that $G/H$ is ICC \cite[Proposition 2.2]{frisch2018non} and that for any irreducible measure $\mu$, the Poisson boundaries of $(G,\mu)$ and $(G/H, \pi_*\mu)$ coincide (see \cite[Proposition 5.1]{erschler2023arboreal}, \cite[Lemma 4.7]{jaworski2004countable}), where $\pi: G \to G/H$ is the quotient map. Given a randomized stopping time $\tau$ for the $\pi_*\mu$-random walk on the quotient group, we may pull back to $G$ to get a randomized stopping time for the $\mu$-random walk.

To deduce Theorem \ref{thm:always-bigger}, it suffices to exhibit a bounded $ \mu _{\tau}$-harmonic function on $ G $ that is not $ \mu $-harmonic. 
Recall that we can sample the $i$th step of our $ \mu _{\tau} $ random walk by choosing $X _{i} \sim p$ and then choosing $ w _{i} \sim \mu _{\tau _{X_i}} $.

Given a natural number $n \in \mathbb{N}$, we can construct a $ \mu _{\tau} $-harmonic function as follows: given the limiting boundary point $ \xi \in \partial \mathcal{F} $ we ask what is the conditional probability that the first step of the $ \mu _{\tau} $-random walk was drawn from the measure $ \mu _{\tau _{n}} $. This produces a bounded function on $ \partial\mathcal{F} $, which induces a $ \mu _{\tau} $-harmonic function on $ G $ by integrating against the $ \mu _{\tau}$-stationary hitting measure $ \nu $. Explicitly, this produces the $\mu_\tau$-harmonic function $f_n(g) = \int _{\partial \mathcal{F}} \mathbb{P} ^{\xi} (X _{1} = n) d g _{*} \nu(\xi)$. Here, $ \left(\mathbb{P} ^{\xi}\right)_{\xi \in \partial \mathcal{F}} $ are the measures coming from the disintegration $ \mathbb{P} = \int _{\partial \mathcal{F}} \mathbb{P} ^{\xi} d \nu (\xi) $ \cite{rohlin1967lectures, kaimanovich2000poisson}. However, we find it convenient to use a slightly different $\mu_\tau$-harmonic function.

If the supports of $ \left(\mu _{\tau _{i}}\right) _{i \geq 0} $ were pairwise disjoint, then we could express these conditional probabilities using the Doob $h$-transform to see that 
 \begin{align*} 
	\mathbb{P} ^{\xi}(X_1=n) &= \sum_{y \in \supp \mu _{\tau _{n}}}\mathbb{P} ^{\xi} (w_1=y) \\
				 &= \sum_{y \in G} p(n) \frac{d y _{*} \nu}{d \nu} (\xi) \mu _{\tau _{n}} (y)
  \end{align*}

The supports of $ \left(\mu _{\tau _{i}}\right) _{i \geq 0} $ need not be disjoint (say, if $\mu_0$ is fully supported), though the latter expression still allows us to define a bounded $\mu_\tau$-harmonic function. Explicitly, our choice of bounded $\mu_\tau$-harmonic function is 
\[ h _{n} (g) := p(n) \cdot \int _{\partial \mathcal{F} } \sum_{y \in G} \frac{dy_*\nu}{d\nu}(\xi) \mu _{\tau _{n}} (y) dg_*\nu(\xi) .\]



  Observe that $ h _{n} (e) = p(n) $ as $\nu$ is $\mu_\tau$-stationary. If $ h_n $ was also $ \mu $-harmonic, then the Optional Stopping Theorem would imply that the average $ \sum_{g \in G} \mu _{\tau _{n}} (g) h_n(g) $ is also equal to $ p(n) $. We will show that this average is at least $ 1-o _{n} (1) $, which is strictly greater than $ p(n) $ for $ n $ sufficiently large. This will be a consequence of the following proposition, proven in Section \ref{section:conditionally-stable}. Recall the sets $S_{n,n}$ defined in Section \ref{section:construction-of-mutau}.
  
\begin{proposition}\label{lemma:conditionally-stable}
	Let $ \mathcal{F} _{n} $ be the collection of subtrees consisting of descendants of elements of the form $ gw $ where $ g \in S _{n,n} $ and $ w \in \Delta _{n} ^{\Phi(n)} $. Then for any $ g \in S _{n,n} $, we have $ g _{*} \nu ( \partial \mathcal{F} _{n}) \geq 1- o_n(1) $.
	  \end{proposition}

Given this proposition, we perform a little bit of averaging.
Expanding the definition of $ h_n(g)$, swapping a sum and integral, and changing variables gives us
\[ \sum_{g \in G} \mu _{\tau _{n}} (g) h_n(g) = p(n) \cdot \int _{\partial \mathcal{F}} \left(\sum_{g' \in G} \mu _{\tau _{n}} (g') \frac{dg' _{*} \nu}{d\nu}(\xi)\right)  \left(\sum_{g \in G} \mu _{\tau _{n}} (g) \frac{dg _{*} \nu}{d\nu}(\xi)\right) d\nu(\xi) .\] 

Writing $ \nu _{n} = \mu _{\tau _{n}} *\nu $, the right side is none other than $ p(n) \cdot \norm{\frac{d\nu _{n}}{d\nu}} ^{2} _{2} $, so our goal is to show that $ \norm{\frac{d\nu _{n}}{d\nu} } _{2} ^{2} \geq (1-o _{n} (1))/p(n) $.

With probability $ (1-o_n(1))p(n) $ the first increment of the $ \mu _{\tau} $-random walk is drawn from $ S_{n,n}$, which by Proposition \ref{lemma:conditionally-stable} implies that $ \nu( \partial \mathcal{F} _{n}) \geq (1-o_n(1))p(n)$. Let $ f _{n} = \frac{d\nu _{n}}{d\nu} $. Then also from Proposition \ref{lemma:conditionally-stable}, we have $ \norm{f _{n} | _{\partial \mathcal{F} _{n}}}_1 \geq 1-o_n(1) $.

Applying Cauchy-Schwarz to $ f_n\cdot 1_{\partial \mathcal{F}_{n}} $, we see that \[ 1-o_n(1) \leq  \norm{f _{n} | _{\partial \mathcal{F} _{n}}}_1 \leq \norm{f_n} _{2} \norm{1 _{\mathcal{F} _{n}}} _{2}  = \norm{f_n} _{2} (1-o_n(1)) \sqrt{p (n)},\] which implies $ \norm{f_n} _{2} ^{2} \geq (1-o_n(1))/p(n)  $.

As a result, we have \[ \sum_{g \in G} \mu _{\tau _{n}} (g) h_n(g) > h_n(e) \]

for $n$ sufficiently large, so we can conclude that the $\mu_\tau$-harmonic function $h_n$ is not $\mu$-harmonic. Therefore the Poisson boundary of $(G,\mu_\tau)$ is strictly larger than that of $(G,\mu)$.
\qed

\begin{remark}
    It is natural at this point to ask for which bounded functions $f: G \to \mathbb{R}$ there exists a randomized stopping time $\tau$ such that $f$ is $\mu_\tau$-harmonic. One can modify our argument, letting $p(0)$ be sufficiently small and replacing $\Phi$ with a sufficiently quickly growing function, to show that with probability close to 1, the $\mu_\tau$-random walk started at any fixed $g$ will forever lie in a subtree of the $(\Phi, \Sigma, A)$-ladder forest rooted at $g$. As a result, one can express any bounded function on $G$ as the pointwise limit of a sequence of bounded functions which are harmonic with respect to some randomized stopping time transform of $\mu$.
\end{remark}

\subsection{Proof of Proposition \ref{lemma:conditionally-stable}}\label{section:conditionally-stable}

We start with the following lemma about record times.

\begin{lemma}\label{lemma:record-skipping}
	Let $ p $ be a probability measure on $ \mathbb{N} $ which has eventually simple records. Then $ \mathbb{P}(n \in (R_{_k})_{k \geq 1} ) \to 0 $ as $ n \to \infty $. 
	  \end{lemma}
	  \begin{proof}
		  If not, then there exists some $ \delta>0 $ and infinitely many $ n_1 < n_2 < ... $ such that $ \mathbb{P}(n_i \in \left( R _{k} \right) _{k \geq 1}  ) \geq \delta$. For each such $ n_i $, with probability at least $ \delta ^{2} $, there are two consecutive record times $ T _{k} < T _{k+1} $ such that $ R _{k} = n_i = R _{k+1} $. Hence with positive probability infinitely many $ n_i $ are non-simple records.
	  \end{proof}

\begin{proof}[Proof of Proposition \ref{lemma:conditionally-stable}]
	Given $ \varepsilon > 0 $, let $ (X_i) _{i \geq 1} $ be i.i.d. samples from $ p $ and let $ (g _{i}) _{i \geq 1} $ independent with $g_i \sim \mu _{\tau_{X_i}} $. For $ k_0 $ sufficiently large depending on $ \varepsilon$ and $ n $ sufficiently large depending on $ k_0$, the following holds with probability at least $ 1- \varepsilon $.

	\begin{enumerate} 
		\item $ n $ is not a record value (by Lemma \ref{lemma:record-skipping})
		\item at least one of $ X_1, ..., X _{\sigma} $ is equal to $ 0 $, where $ \sigma = \inf\{t \geq 1, X_t \geq n\}$ (as $ p(0)>0 $).
		\item For all $ k \geq k_0 $, the $ k $th record $ R _{k} $ is simple and $ T _{k+1} \leq \Phi(R_{k}) $ (by Lemma \ref{lemma:simple-records}).
		\item For all $ k \geq k_0 $, for all $ 1 \leq i \leq \Phi(R_k) $, the increment $ g _{i} $ is drawn from $ S _{X_i, k} $.
		\item $ n $ is larger than $ R_{k_0} $.
		  \end{enumerate}

		  Let $ \{ w _{m}  \} _{m \geq 1}  $ be a sample path on this event. We claim that for every $ m \geq 1 $ and $g \in S_{n,n}$, the element $ g w _{m} $ is spiked and lies in $ \mathcal{F}_n$. Pick $ k$ such that $ R_k < n < R_{k+1} $. 

Then for times $ m \leq T_{k+1}-1 $, since $ w_m $ is a product of at most $ \Phi(R_k) \leq \Phi(n) $ increments drawn from $ \Delta_{R_k} \subset \Delta_n $, the element $ gw_m $ is spiked with spike $ g $, hence it lies in $ \mathcal{F}_{n}$.

For times $ m \geq T_{k+1} $, choose $ k' $ minimal such that $T_{k'} \leq m \leq T _{k'+1} - 1 $. As at least one of $ g_1,...,g_{T_{k'}-1} $ is trivial, then $ gw_{T_{k'}} $ is a product of at most $ \Phi(R_{k'}) $ increments drawn from $ \Delta_{R_{k'}}$, so that $ gw_m $ is still spiked with prefix $ gw_{T_{k'}-1} $ and so $ gw_m $ lies in $ \mathcal{F}_n$.
\end{proof}

\subsection*{Acknowledgments}
This work was done during a visit of the first author to UC San Diego. We would like to thank Inhyeok Choi, Eduardo Silva, and Anush Tserunyan for helpful comments, and would especially like to thank Vadim Kaimanovich for a great many helpful suggestions. We would also like to thank Behrang Forghani for sharing references on the relation between stopping times and the Poisson boundary. The first author was supported by a PGS-D Award from NSERC. The second author was supported by NSF Grant DMS 2348981.

\end{document}